\documentclass [a4paper, 11pt]{article}

\usepackage{amsmath}
\usepackage{amsfonts}
\usepackage{amssymb}
\usepackage{amsthm}
\usepackage{makeidx}
\usepackage{graphicx}
\usepackage{pb-diagram}
\usepackage{epstopdf}
\usepackage{fancyhdr}
\usepackage[all]{xy}

\newtheorem{cor}{Corollary}[section]
\newtheorem{te}[cor]{Theorem}
\newtheorem{p}[cor]{Proposition}
\newtheorem{q}[cor]{Question}
\newtheorem{lemma}[cor]{Lemma}
\theoremstyle{definition}
\newtheorem{de}[cor]{Definition}
\theoremstyle{remark}
\newtheorem{ob}[cor]{Observation}

\newtheorem{nt}[cor]{Notation}

\newcommand{\cz}{\mathbb{C}}
\newcommand{\nz}{\mathbb{N}}

\newcommand{\bb}{\mathcal{B}}

\newcommand{\nr}{\mathcal{N}}
\newcommand{\unit}{\mathcal{U}}

   {\begin{verse}%
     \small }%
   {\end{verse}}

\begin{document}

\title{A Convex Structure on Sofic Embeddings}
\maketitle
\begin{center}
Liviu P\u aunescu\footnote{Work supported by the Sinergia grant
CRSI22-130435 of the Swiss National Science foundation.}
\end{center}

 $\mathbf{Abstract.}$ In \cite{Br} Nathanial Brown
introduced a convex-like structure on the set of unitary
equivalence classes of unital *-homomorphisms of a separable type
$II_1$ factor into $R^\omega$ (ultrapower of the hyperfinite
factor). The goal of this paper is to introduce such a structure
on the set of sofic representations of groups. We prove that if
the commutant of a representation acts ergodicaly on the Loeb
measure space then that representation is an extreme point.

\tableofcontents

\section{Introduction}

Abstract convex-like structures are defined in \cite{Br}. It is
shown that $\mathbb{H}om(N,R^\omega)$ posses such a structure,
where $N$ is a type $II_1$ factor and $R^\omega$ is the ultrapower
of the hyperfinite factor. An important result is that
$[\pi]\in\mathbb{H}om(N,R^\omega)$ is an extreme point iff the
relative commutant $\pi(N)'\cap R^\omega$ is a factor.

We are interested in a similar convex structure on the set of
sofic representations of a group, that we denote by
$Sof(G,P^\omega)$, where $P$ stands for permutations. The role of
von Neumann algebras will be replaced in this paper by ergodic
theory.

We start by briefly introducing the objects that we are working
with: sofic groups, Loeb measure space, the space of sofic
representations. In the second section we recall what an abstract
convex-like structure is and introduce such a structure on
$Sof(G,P^\omega)$. Then we prove our main result that ergodicity
of the commutant of the sofic representation acting on the Loeb
measure space implies that the representation is an extreme point.

\subsection{Ultraproducts of matrix algebras}

Let $\omega$ be a free ultrafilter on $\nz$. We shall work with
ultraproducts of matrix algebras, which is a particular case of
ultraproducts of von Neumann algebras.

Denote by $M_n(\cz)$ or simply $M_n$ the matrix algebra in
dimension $n$. Recall that for $x\in M_n$ we have the trace norm:
$||x||_2=\left(\frac1nTr(x^*x)\right)^{1/2}$ (we normalize it such
that $||Id||_2=1$ independent of the dimension).

Let $(n_k)_k\subset\nz$ be a sequence such that
$\lim_{k\to\infty}n_k=\infty$. Define:
\begin{align*}
&l^\infty(\nz,M_{n_k})=\{x=(x_k)_k\in\Pi_kM_{n_k}:\sup_k||x_k||<\infty\}\\
&\nr_\omega=\{x\in
l^\infty(\nz,M_{n_k}):\lim_{k\to\omega}||x_k||_2=0\}\mbox{ and }\\
&\Pi_{k\to\omega}M_{n_k}=l^\infty(\nz,M_{n_k})/\nr_\omega.
\end{align*}
The ultraproduct $\Pi_{k\to\omega}M_{n_k}$ is a von Neumann
algebra, though the proof is a little involved. If $x_k\in
M_{n_k}$ we shall denote by $\Pi_{k\to\omega}x_k$ the
corresponding element in the ultraproduct. Note that this algebra
has a faithful trace, namely $Tr(x)=\lim_{k\to\omega}Tr(x_k)$,
where $x=\Pi_{k\to\omega}x_k$.

\subsection{The Loeb measure space}

The Loeb space was introduced in \cite{Lo} (our exposition is from
\cite{El-Sz}). We shall denote by $P_n$ the subgroup of
permutation matrices and by $D_n$ the subalgebra of diagonal
matrices. We shall interpret $\Pi_{k\to\omega}P_{n_k}$ and
$\Pi_{k\to\omega}D_{n_k}$ as subsets in $\Pi_{k\to\omega}M_{n_k}$.
By a theorem of Sorin Popa (see \cite{Po1}, Proposition 4.3)
$\Pi_{k\to\omega}D_{n_k}$ is a maximal abelian nonseparable
subalgebra of $\Pi_{k\to\omega}M_{n_k}(\cz)$. It is isomorphic to
$L^\infty(X)$ where $X$ is a Loeb measure space. The construction
of the Loeb space is valid for any sequence of probability spaces,
but we shall need it just for finite spaces.

Let $(X_{n_k},\mu_{n_k})$ be a space with $n_k$ points equipped
with the normalized counting measure such that $(D_{n_k},Tr)\simeq
L^\infty(X_{n_k},\mu_{n_k})$. Let $p\in\Pi_{k\to\omega}D_{n_k}$ be
a projection. It is not difficult to see that
$p=\Pi_{k\to\omega}p_k$, where $p_k$ is a projection in $D_{n_k}$
and $Tr(p)=\lim_{k\to\omega}Tr(p_k)$ by definition. A projection
in the algebra is the same information as a measurable subset in
the underlaying space. This discussion offers a picture of how the
Loeb space should be constructed.

Let $(X_{n_k})_\omega$ be the algebraic ultraproduct, i.e.
$(X_{n_k})_\omega=\Pi_kX_{n_k}/\sim_\omega$, where $\Pi_kX_{n_k}$
is the Cartesian product and $(x_k)\sim_\omega(y_k)$ iff
$\{k:x_k=y_k\}\in\omega$. If $x_k\in X_{n_k}$ we shall denote by
$(x_k)_\omega$ the corresponding element in $(X_{n_k})_\omega$. If
$A_k\subset X_{n_k}$ then define
$(A_k)_\omega=\{(x_k)_\omega:\{k:x_k\in
A_k\}\in\omega\}\subset(X_{n_k})_\omega$. Let
$\bb_\omega^0=\{(A_k)_\omega:A_k\subset X_{n_k}\}$. Then
$\bb_\omega^0$ is a Boolean algebra of subsets of
$(X_{n_k})_\omega$.

For $(A_k)_\omega\in\bb_\omega^0$ define
$\mu_\omega((A_k)_\omega)=\lim_{k\to\omega}\mu_{n_k}(A_k)$. Let
$\bb_\omega$ be the completion of $\bb_\omega^0$ w.r.t the measure
$\mu_\omega$. Then $((X_{n_k})_\omega,\bb_\omega,\mu_\omega)$ is a
nonseparable probability space and $L^\infty((X_{n_k})_\omega,
\mu_\omega)\simeq(\Pi_{k\to\omega}D_{n_k},Tr)$.

\subsection{Sofic groups}

Introduced by Gromov in \cite{Gr}, sofic groups have received
considerable attention in the last years. For an introduction to
the subject see the nice survey articles of Vladimir Pestov,
\cite{Pe},\cite{Pe-Kw}.

\begin{de}
A group $G$ is called \emph{sofic} if there exists a sequence
$\{n_k\}_k\subset\nz$, $\lim_k n_k=\infty$ and an injective group
morphism $\Theta:G\to\Pi_{k\to\omega}P_{n_k}$
\end{de}

The sequence $(n_k)_k$ doesn't have a special role. If such a
morphism exists for some $(n_k)_k$ it will exists for any other
$(m_k)_k$ as long as $\lim_k m_k=\infty$. The following theorem is
due to Gabor Elek and Endre Szabo, \cite{El-Sza}.

\begin{te}
A group $G$ is sofic iff there exists a group morphism
$\Theta:G\to\Pi_{k\to\omega}P_{n_k}$ such that $Tr(\Theta(g))=0$
for any $g\neq e$.
\end{te}
\begin{proof}
A morphism $\Theta$ such that $Tr(\Theta(g))=0$ for any $g\neq e$
is clearly injective. For the reverse implication let
$\Theta:G\to\Pi_{k\to\omega}P_{n_k}$ be an injective morphism. If
$|Tr(\Theta(g))=1|$ then $Tr(\Theta(g))=1$ and $g=e$. In the end
we have $|Tr(\Theta(g))|<1$ for any $g\neq e$.

Construct $\Theta^{(m)}=\Theta\otimes\Theta\otimes\ldots
\otimes\Theta$ ($m$ times tensor product), i.e.
$\Theta^{(m)}(g)=\Pi_{k\to\omega}u_g^k\otimes u_g^k\otimes\ldots
\otimes u_g^k$, where $\Theta(g)=\Pi_{k\to\omega}u_g^k$. This is a
representation of $G$ on $\Pi_{k\to\omega}P_{n_k^m}$. Then
$Tr(\Theta^{(m)}(g))= Tr(\Theta(g))^m$. This means that
$Tr(\Theta^{(m)}(g)) \to_{m\to\infty}0$ for $g\neq e$. A diagonal
argument will finish the proof.
\end{proof}

\begin{de}
A \emph{sofic representation of G} is a group morphism
$\Theta:G\to\Pi_{k\to\omega}P_{n_k}$ such that $Tr(\Theta(g))=0$
for any $g\neq e$, where $\{n_k\}_k$ is any sequence of natural
numbers such that $n_k\to_{k\to\infty}\infty$.
\end{de}

\begin{nt}
Let $\Theta:G\to\Pi_{k\to\omega}P_{n_k}$ be a group morphism,
$\Theta=\Pi_{k\to\omega}\theta_k$. Let $\{r_k\}_k$ be a sequence
of natural numbers. Define $\Theta\otimes
1_{r_k}:G\to\Pi_{k\to\omega}P_{n_kr_k}$, $\Theta\otimes
1_{r_k}=\Pi_{k\to\omega}\theta_k\otimes 1_{r_k}$. We shall call
$\Theta\otimes 1_{r_k}$ an \emph{amplification} of $\Theta$.
\end{nt}

\begin{nt}\label{direct sum}
There is also a \emph{direct sum} of two sofic representations. If
$\Theta:G\to\Pi_{k\to\omega}P_{n_k}$, $\Theta=\Pi_{k\to\omega}
\theta_k$ and $\Phi:G\to\Pi_{k\to\omega}P_{m_k}$,
$\Phi=\Pi_{k\to\omega}\phi_k$ then define
$\Theta\oplus\Phi:G\to\Pi_{k\to\omega}P_{n_k+m_k}$ by
$\Theta\oplus\Phi=\Pi_{k\to\omega} \theta_k\oplus\phi_k$.
\end{nt}

We shall need the following lemma from \cite{Pa}. We also include
a short proof for the reader's convenience.

\begin{lemma}\label{permutations}
Let $\{e_i|i\in\nz\}$ be projections in $\Pi_{k\to\omega}D_{n_k}$
such that $\sum_ie_i=1$. Let $\{u_i|i\in\nz\}$ be unitary elements
in $\Pi_{k\to\omega}P_{n_k}$ such that $v=\sum_ie_iu_i$ is a
unitary. Then $v\in\Pi_{k\to\omega}P_{n_k}$.
\end{lemma}
\begin{proof}
Using the equation $\sum_ie_i=1$ we can construct projections
$e_i^k\in D_{n_k}$ such that:
\begin{enumerate}
\item $e_i=\Pi_{k\to\omega}e_i^k$; \item $\sum_ie_i^k=1_{n_k}$.
\end{enumerate}
By hypothesis we have $u_i=\Pi_{k\to\omega}u_i^k$ where $u_i^k\in
P_{n_k}$. If $v^k=\sum_ie_i^ku_i^k$ then $v=\Pi_{k\to\omega}v^k$,
but $v^k$ are not necessary unitary matrices. However $v^k$ is
still a matrix only with $0$ and $1$ entries and exactly one entry
of $1$ on each row.

We need to estimate the number of columns in $v^k$ having only $0$
entries. Denote this number by $r_k$. Then $v^{k*}v^k$ is a
diagonal matrix having $r_k$ entries of $0$ on the diagonal. This
implies:
\[||v^{k*}v^k-Id||_2^2\geq\frac{r_k}{n_k}.\]
Because $\Pi_{k\to\omega}v^{k*}v^k=1$ we have
$r_k/n_k\to_{k\to\omega}0$. We now construct $w^k$ as follows. The
matrix $v^k$ has $n_k-r_k$ columns with at least one nonzero
entry. For each such column $j$ chose a row $i$ such that
$v^k(i,j)=1$. Let $w^k(i,j)=1$. In this way we have $n_k-r_k$
nonzero entries in $w^k$, all of them distributed on different
rows and different columns. Choose a bijection between the
remaining $r_k$ rows and $r_k$ columns and complete $w^k$ to a
permutation matrix. Then:
\[||v^k-w^k||_2^2=\frac{2r_k}{n_k}.\]
Combined with $r_k/n_k\to_{k\to\omega}0$ we get
$v=\Pi_{k\to\omega}w^k$. This will prove the lemma.
\end{proof}

\subsection{The metric space $Sof(G,P^\omega)$}

\begin{de}
For a countable group $G$ define $Sof(G,P^\omega)$ the set of
sofic representations of $G$, factored by the following
equivalence relation: $(\Theta_1:G\to\Pi_{k\to\omega}P_{n_k})\sim
(\Theta_2:G\to\Pi_{k\to\omega}P_{m_k})$ iff there are two
sequences of natural numbers $\{r_k\}_k$ and $\{t_k\}_k$ such that
$n_kr_k=m_kt_k$ for any $k$ and there exists
$u\in\Pi_{k\to\omega}P_{n_kr_k}$ such that $\Theta_2\otimes
1_{t_k}=Ad u\circ(\Theta_1\otimes 1_{r_k})$.
\end{de}

\begin{nt}
For a sofic representation $\Theta$ we shall denote by $[\Theta]$
its class in $Sof(G,P^\omega)$.
\end{nt}

\begin{ob}\label{amenablegroups}
By Theorem 2 of Gabor Elek and Endre Szabo from \cite{El-Sza2},
the space $Sof(G,P^\omega)$ has exactly one point iff the group
$G$ is amenable.
\end{ob}

In order to define a metric on $Sof(G,P^\omega)$ we need to fix a
counting of the group $G$. So let $G=\{g_0,g_1,\ldots\}$ where
$g_0=e$. For $[\Theta],[\Phi]\in Sof(G,P^\omega)$ define:
\[d([\Theta],[\Phi])=inf\{\big(\sum_{i=1}^\infty\frac1{4^i}||(\Theta\otimes
1)(g_i)-p(\Phi\otimes
1)(g_i)p^*||_2^2\big)^\frac12:\{n_k\}_k;p\in\Pi_{k\to\omega}P_{n_k}\}\]

The infimum is taken over all the sequences $\{n_k\}_k$ such that
the two sofic representations $\Theta$ and $\Phi$ have
amplifications in that dimension. It is clear that this definition
does not depend on $\Theta$ and $\Phi$, but only of their classes
in $Sof(G,P^\omega)$.

Due to a diagonal argument we can see that the infimum in the
definition is attain. This implies that $d$ is indeed a distance.
Also $Sof(G,P^\omega)$ is complete with this metric.

\section{The convex structure on the set of sofic embeddings}

Let us first recall what a metric space with a convex-like
structure is. The next section is from \cite{Br}.

\subsection{Metric spaces with a convex-like structure}

Let $(X, d)$ be a complete metric space which is bounded (there is
a constant $C$ such that $d(x,y)\leq C$ for all $x,y\in X$). In
order to provide an abstract convex-like structure on $X$ we need
to define the element $t_1x_1+t_2x_2+\ldots t_nx_n$, where
$x_1,\ldots,x_n\in X$ and $0\leq t_i\leq 1$ such that
$\sum_{i=1}^n t_i=1$. We shall ask for the following axioms (as
Nathanial Brown put it: "properties one would expect if $X$ were
an honest convex subset of a bounded ball in some normed linear
space")
\begin{enumerate}
\item (commutativity)
$t_1x_1+\ldots+t_nx_n=t_{\sigma(1)}x_{\sigma(1)}+\ldots+t_{\sigma(n)}x_{\sigma(n)}$
for every permutation $\sigma\in Sym(n)$; \item (linearity) if
$x_1=x_2$, then $t_1x_1+t_2x_2+\ldots+t_nx_n=(t_1+t_2)x_1+t_3x_3+
\ldots+t_nx_n$; \item (scalar identity) if $t_i = 1$, then $t_1x_1
+\ldots+t_nx_n=x_i$; \item (metric compatibility)
$d(t_1x_1+\ldots+t_nx_n,s_1x_1+\ldots s_nx_n)\leq
C\sum_i|t_i-s_i|$ and $d(t_1x_1+\ldots+t_nx_n,t_1y_1+\ldots+
t_ny_n)\leq\sum_it_id(x_i,y_i)$; \item (algebraic compatibility)
\[t\left(\sum_{i=1}^nt_ix_i\right)+(1-t)\left(\sum_{j=1}^m s_jy_j\right)=\sum_{i=1}^n
tt_ix_i+\sum_{j=1}^m(1-t)s_jy_j.\]
\end{enumerate}

In \cite{Ca-Fr}, Valerio Capraro and Toblias Fritz proved that
this axioms are enough to deduce that $X$ is a closed convex
subset in an abstractly constructed Banach space.

\subsection{Convex combinations of sofic representations}

We now define the convex-like structure on $Sof(G,P^\omega)$.

\begin{de}
Let $n\in\nz$ and for $i=1,2,\ldots,n$ let
$\Theta_i:G\to\Pi_{k\to\omega}P_{m_k^i}$ be a sofic representation
and $\lambda_i\geq0$ such that $\sum_{i=1}^n\lambda_i=1$. Define
$\sum_{i=1}^n\lambda_i\Theta_i$ as follows: choose natural numbers
$r_k^i$ such that $\lim_{k\to\omega}m_k^jr_k^j/\sum_{i=1}^n
m_k^ir_k^i=\lambda_j$ for any $j=1,2,\ldots n$ and set
$\sum_{i=1}^n\lambda_i\Theta_i:G\to P_{\sum_{i=1}^n m_k^ir_k^i}$,
$\sum_{i=1}^n\lambda_i\Theta_i= \oplus_{i=1}^n(\Theta_i\otimes
1_{r_k^i})$.
\end{de}

\begin{p}
$[\sum_{i=1}^n\lambda_i\Theta_i]$ is well defined, i.e. depends
only on $[\Theta_i]$ and $\lambda_i$, $i=1,\ldots,n$. We shall
denote this object by $\sum_{i=1}^n\lambda_i[\Theta_i]$.
\end{p}

\begin{p}
The convex structure defined on $Sof(G,P^\omega)$ obeys the axioms
$(1)-(5)$ of abstract convex-like structures.
\end{p}
\begin{proof}
Verifications are trivial. Maybe the first part of axiom $(4)$ is
a little more technical.
\end{proof}

\subsection{Cutting sofic representations}

A way of constructing new sofic representations out of old ones is
by cutting a representation with a projection from the commutant.
This is actually the reverse operation of the direct sum as
intruduced in \ref{direct sum}. Let
$\Theta:G\to\Pi_{k\to\omega}P_{n_k}$ be a sofic representation,
$\Theta=\Pi_{k\to\omega}\theta_k$, where $\theta_k:G\to P_{n_k}$.
Let also $p\in\Theta(G)'\cap \Pi_{k\to\omega}D_{n_k}$. Choose
projections $p_k\in D_{n_k}$ such that $p=\Pi_{k\to\omega}p_k$.
Then $Tr(p_k)=\frac{m_k}{n_k}$, with $m_k\in\nz$. Define
$\Theta_p:G\to \Pi_{k\to\omega}P_{m_k}$ by
$\Theta_p(g)=\Pi_{k\to\omega}p_k\theta_k(g)p_k$. In fact
$\Theta_p(g)=p\Theta(g)p=p\Theta(g)$. By definition $\Theta_p$
depends on the choice of projections $p_k$, but it is easy to see
that $[\Theta_p]$ does not depend on this choice.

\begin{de}
If $\Theta:G\to\Pi_{k\to\omega}P_{n_k}$ and
$p\in\Theta(G)'\cap\Pi_{k\to\omega}D_{n_k}$ then define
$\Theta_p:G\to p(\Pi_{k\to\omega}P_{n_k})p$ by
$\Theta_p(g)=p\Theta(g)$.
\end{de}

\begin{ob}\label{recoveringembeddings}
If $\Theta=\sum_{i=1}^n\lambda_i\Phi_i$ then there exists
$p\in\Theta(G)'\cap\Pi_{k\to\omega}D_{\sum_{i=1}^n m_k^ir_k^i}$
with $Tr(p)=\lambda_i$ such that $[\Phi_i]=[\Theta_p]$.
\end{ob}

\subsection{Actions on the Loeb space}

If $u\in\Pi_{k\to\omega}P_{n_k}$ then $u(\Pi_{k\to\omega}D_{n_k})
u^*=\Pi_{k\to\omega}D_{n_k}$. Also the reverse is true: if
$x(\Pi_{k\to\omega}D_{n_k})x^*=\Pi_{k\to\omega}D_{n_k}$ then
$x=a\cdot u$, where $u\in\Pi_{k\to\omega}P_{n_k}$ and $a$ is a
unitary element of $\Pi_{k\to\omega}D_{n_k}$. In an operator
language the normalizer of $\Pi_{k\to\omega}D_{n_k}$ is
$\unit(\Pi_{k\to\omega}D_{n_k})\cdot\Pi_{k\to\omega}P_{n_k}$.

As $\Pi_{k\to\omega}D_{n_k}\simeq L^\infty((X_{n_k})_\omega,
\mu_\omega)$, $u$ defines an automorphism of
$((X_{n_k})_\omega,\mu_\omega)$. For a sofic representation we are
interested in the action of the commutant on this space. Such
actions were considered by David Kerr and Hanfeng Li in
\cite{Ke-Li}.

\begin{nt}
If $\Theta:G\to\Pi_{k\to\omega}P_{n_k}$ is a sofic representation
then we shall denote by $\alpha(\Theta)$ the action of
$\Theta(G)'\cap\Pi_{k\to\omega}P_{n_k}$ on the Loeb space
$((X_{n_k})_\omega,\mu_\omega)$.
\end{nt}

The goal of this article is to link the ergodicity of this action
with extreme points in the convex-like structure. We shall now
prove that ergodicity is preserved under amplifications. For this
we need the next lemma.

\begin{lemma}
Let $n,r\in\nz$ and $A_1,\ldots, A_r\subset\{1,\ldots n\}$. Then
exists $A\subset\{1,\ldots,n\}$ and $p\in Sym(r)$ such that:
\[\sum_{j=1}^r|A\triangle A_j|\leq
\sum_{j=1}^r|A_{p(j)}\triangle A_j|.\]
\begin{proof}
For $i=1,\ldots,n$ let $a_i=|\{j:i\in A_j\}|$. Define
$A=\{i:r<2a_i\}$ ($i$ is an element of $A$ iff more than half sets
$A_j$ contain $i$). Then $\sum_{j=1}^r|A\triangle
A_j|=\sum_{i=1}^nmin\{a_i,r-a_i\}$.

For $p\in Sym(r)$ define $R(p)=\sum_{j=1}^r|A_{p(j)}\triangle
A_j|$. We shall try to evaluate $\sum_{p\in Sym(r)}R(p)$. We want
to count how many times $i\in A_{p(j)}$ and $i\notin A_j$.

For $r-a_i$ different values we have $i\notin A_j$; $p(j)$ will be
a given fix value for $(r-1)!$ permutations in $Sym(r)$. For
another $a_i$ of this values we have $i\in A_{p(j)}$. The number
we are looking for is $a_i(r-a_i)(r-1)!$. It may also happen that
$i\notin A_{p(j)}$ and $i\in A_j$. In the end we have:
\[\sum_{p\in Sym(r)}R(p)=\sum_{i=1}^n2a_i(r-a_i)(r-1)!.\]
It follows that there exists $p\in Sym(r)$ such that
$R(p)\geq\sum_{i=1}^n2a_i(r-a_i)/r$. It is easy to see that
$min\{a_i,r-a_i\}\leq 2a_i(r-a_i)/r$ so $R(p)\geq
\sum_{j=1}^r|A\triangle A_j|$.
\end{proof}
\end{lemma}

\begin{p}\label{amplificationpreservesergodicity}
Let $\Theta:G\to\Pi_{k\to\omega}P_{n_k}$ be a sofic representation
and $\{r_k\}$ a sequence of natural numbers. If $\alpha(\Theta)$
is ergodic then also $\alpha(\Theta\otimes 1_{r_k})$ is ergodic.
\end{p}
\begin{proof}
Let $(S_k)_\omega$ be a subset of the Loeb space
$(X_{n_rr_k})_\omega$ such that $u((S_k)_\omega)=(S_k)_\omega$ for
any $u\in(\Theta\otimes 1_{r_k})'\cap\Pi_{k\to\omega} P_{n_kr_k}$.
Assume that $\mu_\omega((S_k)_\omega)\neq\{0,1\}$. We shall regard
$X_{n_kr_k}$ as $r_k$ copies of $X_{n_k}$. If $u_k$ is of the form
$1_{n_k}\otimes p$, with $p\in Sym(\{1,\ldots,r_k\})$ then $u_k$
will permute this $r_k$ copies of $X_{n_k}$.

With respect to this partition of $X_{n_kr_k}$ we have
$S_k=\sqcup_{j=1}^{r_k}A_k^j$ where $A_k^j\subset X_{n_k}$. Apply
the previous lemma to $A_k^1,\ldots,A_k^{r_k}$ to get a set $A_k$
and a permutation $p_k$. Define
$T_k=A_k\times\{1,\ldots,r_k\}\subset X_{n_kr_k}$. Then by the
conclusion of the previous lemma we have
\[|T_k\triangle S_k|\leq |(1_{n_k}\otimes p_k)(S_k)\triangle
S_k|.\] Define $u=\Pi_{k\to\omega}1_{n_k}\otimes p_k$. As
$u((S_k)_\omega)=(S_k)_\omega$ by the previous inequality we have
$(T_k)_\omega=(S_k)_\omega$.

It is easy to see that $\mu_\omega((T_k)_\omega)$ in
$(X_{n_rr_k})_\omega$ is equal to $\mu_\omega((A_k)_\omega)$ in
$(X_{n_r})_\omega$. Now the action $\alpha(\Theta)$ is ergodic. So
there exists $v\in\Theta'\cap\Pi_{k\to\omega}P_{n_k}$ such that
$\mu_\omega(v((A_k)_\omega)\triangle (A_k)_\omega)>0$. Define
$u=v\otimes 1_{r_k}$, $u\in(\Theta\otimes
1_{r_k})'\cap\Pi_{k\to\omega} P_{n_kr_k}$. Again, because we just
have an amplification $\mu_\omega(u((T_k)_\omega)\triangle
(T_k)_\omega)=\mu_\omega(v((A_k)_\omega)\triangle (A_k)_\omega)$.
This contradicts $u((S_k)_\omega)=(S_k)_\omega$ and we are done.
\end{proof}

\subsection{Extreme points in the convex structure}

We now turn our attention to extreme points in the convex
structure. Our first lemma is similar to Proposition 3.3.4 from
\cite{Br}.

\begin{lemma}
Let $[\Theta]\in Sof(G,P^\omega)$. Then $[\Theta]$ is an extreme
point iff for any projection $p\in\Theta(G)'\cap
\Pi_{k\to\omega}D_{n_k}$ $[\Theta]=[\Theta_p]$.
\end{lemma}
\begin{proof}
If $p\in\Theta(G)'\cap \Pi_{k\to\omega}D_{n_k}$ then
$\Theta=\Theta_p\oplus\Theta_{1-p}$. Then, by the definition of
the convex structure $[\Theta]=Tr(p)[\Theta_p]+
(1-Tr(p))[\Theta_{1-p}]$. If $[\Theta]$ is an extreme point then
$[\Theta]=[\Theta_p]$ for any $p$.

The reverse implication is a consequence of
(\ref{recoveringembeddings}): if $\Theta=\lambda_1\Phi_1+
\lambda_2\Phi_2$ then there exists
$p\in\Theta(G)'\cap\Pi_{k\to\omega}D_{n_k}$ with $Tr(p)=\lambda_1$
such that $[\Phi_1]=[\Theta_p]$.
\end{proof}

\begin{te}\label{extremepoints}
Let $\Theta:G\to\Pi_{k\to\omega}P_{n_k}$ be a sofic
representation. Assume that the action $\alpha(\Theta)$ on the
Loeb space $((X_{n_k})_\omega,\mu_\omega)$ is ergodic. Then
$[\Theta]$ is an extreme point.
\end{te}
\begin{proof}
If $[\Theta]$ is not an extreme point then there exists
$\Theta_1,\Theta_2:G\to\Pi_{k\to\omega}P_{m_k}$ such that
$[\Theta_1]\neq[\Theta_2]$ and $\Theta\otimes
1=\Theta_1\oplus\Theta_2$. By hypothesis and Proposition
\ref{amplificationpreservesergodicity} $\alpha(\Theta\otimes 1)$
is ergodic.

Let $\{p_i\}_i\subset(\Theta_1(G)'\cap\Pi_{k\to\omega}D_{m_k})$
and $\{q_i\}_i\subset(\Theta_2(G)'\cap\Pi_{k\to\omega}D_{m_k})$ be
a maximal family of disjoint projections such that
$(\Theta_2)_{q_i}=Ad u_i\circ(\Theta_1)_{p_i}$ (no tensor, this
implies $Tr(p_i)=Tr(q_i)$ as matrix dimensions have to be the
same) where $u_i$ is an ultraproduct of permutations. Assume that
$\sum_ip_i<1$ in $\Pi_{k\to\omega}D_{m_k}$.

Let $p=1-\sum_ip_i$ and $q=1-\sum_iq_i$. Let $\tilde p,\tilde
q\in(\Theta(G)\otimes 1)'\cap\Pi_{k\to\omega}D_{2m_k}$ defined by
$\tilde p=p\oplus 0$, $\tilde q=0\oplus q$. Let $A_p,A_q$ be the
subsets of $(X_{2m_k})_\omega$ corresponding to the projections
$\tilde p$ and $\tilde q$. Because the action of
$(\Theta(G)\otimes 1)'\cap\Pi_{k\to\omega}P_{2m_k}$ is ergodic
there exists $u\in(\Theta(G)\otimes
1)'\cap\Pi_{k\to\omega}P_{2m_k}$ such that $\mu_\omega(u(A_p)\cap
A_q)>0$. This is equivalent to $\tilde qu\tilde p\neq 0$.

Let $v=\tilde qu\tilde p$, $p_0=v^*v=\tilde p(u^*\tilde qu)$ and
$q_0=vv^*=\tilde q(u\tilde pu^*)$. Then:
\[v(p_0(\Theta\otimes 1)p_0)v^*=vp_0v^*(\Theta\otimes
1)=q_0(\Theta\otimes 1)=q_0(\Theta\otimes 1)q_0.\] This implies
$v((\Theta_1)_{p_0}\oplus 0)v^*=0\oplus(\Theta_2)_{q_0}$, so
families $\{p_i\}$ and $\{q_i\}$ are not maximal. Then we must
have $\sum_ip_i=1=\sum_iq_i$. Recall that $(\Theta_2)_{q_i}=Ad
u_i\circ(\Theta_1)_{p_i}$.

Define $u=\sum_i u_ip_i$. Then $u\in\Pi_{k\to\omega}P_{m_k}$ by
\ref{permutations} and $\Theta_2=Ad u\circ\Theta_1$, contradicting
$[\Theta_1]\neq[\Theta_2]$.
\end{proof}

Using this theorem (and Theorem 2 from \cite{El-Sza2}, see also
observation \ref{amenablegroups}) we can construct sofic
representations such that the commutant acts non-ergodicaly for
any sofic non-amenable group.

\begin{q}
Is a converse of Theorem \ref{extremepoints} also true?
\end{q}

\subsection{Examples of extreme points}

Theorem \ref{extremepoints} allows us to provide some examples of
extreme points. David Kerr and Hanfeng Li proved that
$\alpha(\Theta)$ is ergodic for any $\Theta$ when $G$ is amenable
(\cite{Ke-Li}, Theorem 5.8). It follows that any element of
$Sof(G,P^\omega)$ is an extreme point. This is possible only if
$Sof(G,P^\omega)$ consists of one point, which is consistent with
results from \cite{El-Sza2}. The proof of Theorem 5.8 from
\cite{Ke-Li} contains something more.

\begin{p}(Proof of Theorem 5.8,\cite{Ke-Li})\label{sameamenable} There exists
$f:(0,1)\to(0,1)$ a continuous function such that, for any
amenable group $H$, for any sofic representation
$\Theta:H\to\Pi_{k\to\omega}P_{n_k}$ and any $Y=(Y_k)_\omega$,
$Z=(Z_k)_\omega$ subsets of $(X_{n_k})_\omega$ of strictly
positive measure there exists $u\in\Theta'(H)\cap
\Pi_{k\to\omega}P_{n_k}$ such that:
\[\mu_\omega(u(Y)\cap Z)\geq
f(min(\mu_\omega(Y),\mu_\omega(Z))).\]
\end{p}

Using this proposition we can prove the existence of extreme
points for initially sub-amenable groups.

\begin{te}
Let $G$ be an initially sub-amenable group. Then there exists
$\Theta$ a sofic representation of $G$ such that $\alpha(\Theta)$
is ergodic.
\end{te}
\begin{proof}
For this proof we need product ultrafilter techniques (see
\cite{Ca-Pa}). The drawback of the proof is that we start with an
ultrafilter $\omega$ and get in the end an extreme point for
$\omega\otimes\omega$, in other words an extreme point in the set
$Sof(G,P^{\omega\otimes\omega})$.

Let $G=\cup_kF_k$, where $\{F_k\}_k$ is an increasing sequence of
finite subsets of $G$. By hypothesis for each $k$ there exists
$H_k$ an amenable group and $\phi_k:F_k\to H_k$ such that
$\phi(g)\phi(h)=\phi(gh)$ when $g,h,gh\in F_k$. Choose
$\Theta_k:H_k\to\Pi_{r\to\omega}P_{n_k,r}$ a sofic representation
of $H_k$, $n_{k,r}\in\nz$.

The ultraproduct $\Pi_{r\to\omega}P_{n_k,r}$ is a metric group for
each $k$. We can construct the ultraproduct of these metric groups
to get the equality:
\[\Pi_{k\to\omega}(\Pi_{r\to\omega}P_{n_k,r})=
\Pi_{(k,r)\to\omega\otimes\omega}P_{n_k,r}.\] Define
$\Theta:G\to\Pi_{(k,r)\to\omega\otimes\omega}P_{n_k,r}$ by:
\[\Theta(g)=\Pi_{k\to\omega}\Theta_k(\phi_k(g)).\]
We now prove that $\alpha(\Theta)$ is ergodic. Let
$Y=(Y_{k,r})_{\omega\otimes\omega}$,
$Z=(Z_{k,r})_{\omega\otimes\omega}$ be subsets of
$(X_{n_{k,r}})_{\omega\otimes\omega}$ of strictly positive
measure. Define $Y_k=((Y_{k,r})_r)_\omega$ and
$Z_k=((Z_{k,r})_r)_\omega$. Then $\mu_{\omega\otimes\omega}(Y)=
\lim_{k\to\omega}\mu_\omega(Y_k)$ and the same for $Z$. Because
$H_k$ is amenable, by Proposition \ref{sameamenable}, there exists
$u_k\in\Theta_k'(H_k)\cap \Pi_{r\to\omega}P_{n_{k,r}}$ such that:
\[\mu_\omega(u_k(Y_k)\cap Z_k)\geq
f(min(\mu_\omega(Y_k),\mu_\omega(Z_k))).\] Let
$u=\Pi_{k\to\omega}u_k$. Then $u$ commutes with $\Theta$ and by
continuity:
\[\mu_{\omega\otimes\omega}(u(Y)\cap Z)\geq
f(min(\mu_{\omega\otimes\omega}(Y),\mu_{\omega\otimes\omega}(Z)))>0.\]
\end{proof}

In \cite{Cor}, Yves Cornulier constructed a sofic group that is
not initially sub-amenable, so the last theorem doesn't solve the
problem of existence of extreme points for sofic groups in
general.

Residually finite groups are initially sub-amenable, but in this
case there is an easier way of constructing extreme points.

\begin{te}(Theorem 5.7,\cite{Ke-Li})
Let $G$ be a residually finite group and let $\{G_i\}_{i\in\nz}$
be a sequence of finite index normal subgroups such that
$\cap_{n\in\nz}\cup_{i\geq n}G_i=\{e\}$. Let $\Theta$ be the sofic
representation of $G$ constructed by taking the left action of $G$
on $G/G_i$. Then $\alpha(\Theta)$ is ergodic.
\end{te}

\section*{Acknowledgement}

I am grateful to Hanfeng Li for his nice talk at the LaWiNe
seminar, where he spoked about the ergodicity of the commutant of
the sofic group acting on the Loeb measure space that inspired
this work, and for remarks on a previous version of this paper. I
thank Gabor Elek for useful discussions on the topics of the
article.

LIVIU P\u AUNESCU, \emph{UNIVERSITY of VIENNA and INSTITUTE of
MATHEMATICS "S. Stoilow" of the ROMANIAN ACADEMY}(on leave) email:
liviu.paunescu@imar.ro


\begin{thebibliography}{11}
\bibitem[Br]{Br} N. Brown, \emph{Topological dynamical systems associated to
$II_1$-factors}, Advances in Mathematics, Volume 227, Issue 4,
Pages 1665-1699.
\bibitem [Ca-Fr]{Ca-Fr} V. Capraro - T. Fritz, \emph{On the
axiomatization of convex subsets of Banach spaces}, to appear in
Proceedings of the AMS.
\bibitem[Ca-P\u a]{Ca-Pa} V. Capraro - L. P\u aunescu, \emph{Product Between Ultrafilters and Applications
 to the Connes' Embedding Problem} J.Oper.Theory, Volume 68, Issue 1, pag. 165-172.
\bibitem [Cor]{Cor} Y. Cornulier, \emph{A sofic group away from amenable groups},
arXiv:0906.3374.
\bibitem [El-Sz1]{El-Sza} G. Elek - E. Szabo, \emph{Hyperlinearity, essentially free actions
and L2-invariants. The sofic property}, Math. Ann. 332 (2005), no.
2, 421-441.
\bibitem [El-Sz2]{El-Sza2} G. Elek - E. Szabo, \emph{Sofic representations of amenable
groups}, Proc. Amer. Math. Soc. 139 (2011), pag. 4285-4291.
\bibitem [El-Sze]{El-Sz} G. Elek - B. Szegedy, \emph{Limits of hypergaphs, removal and
regularity lemmas. A non-standard approach}, arXiv:0705.2179v1.
\bibitem [Gr]{Gr} M. Gromov, \emph{Endomorphism of symbolic
algebraic varieties}, J. Eur. Math. Soc. 1 (1999) 109-197.
\bibitem[Ke-Li]{Ke-Li} D. Kerr - H. Li., \emph{Combinatorial independence and
sofic entropy}, arXiv:1208.2464.
\bibitem [Lo]{Lo} P. E. Loeb, \emph{Conversion from nonstandard to standard measure
spaces and applications in probability theory} Trans. Amer. Math.
Soc. 211 (1975), 113–122.
\bibitem[P\u a]{Pa} P\u aunescu, L; \emph{On Sofic Actions and Equivalence Relations};
Journal of Functional Analysis Volume 261, Issue 9, Pages
2461-2485.
\bibitem [Pe]{Pe}V. Pestov, \emph{Hyperlinear and sofic groups: a brief guide}, Bull.
Symbolic Logic 14 (2008) no. 4, 449–480.
\bibitem [Pe-Kw]{Pe-Kw} V. Pestov - A. Kwiatkowska, \emph{An introduction to hyperlinear and sofic groups},
arxiv:0911.4266 (2009).
\bibitem [Po]{Po1}S. Popa, \emph{On a Problem of R.V. Kadison on Maximal Abelian *-Subalgebras in
Factors}, Inv. Math., 65 (1981), 269-281.
\end{thebibliography}
\end{document}